\documentclass[12pt, reqno]{amsart}
\usepackage{amsmath, amsthm, amscd, amsfonts, amssymb, graphicx, color}
\usepackage[bookmarksnumbered, colorlinks, plainpages]{hyperref}

\input{mathrsfs.sty}

\newtheorem{theorem}{Theorem}[section]
\newtheorem{lemma}[theorem]{Lemma}

\newtheorem{corollary}[theorem]{Corollary}
\theoremstyle{definition}

\theoremstyle{remark}
\newtheorem{remark}[theorem]{Remark}
\numberwithin{equation}{section}
\begin{document}

\title [   Improvements of   operator   inequality  ]{  Improvements of   some   operator   inequalities   involving    positive  linear  maps
via   the Kantorovich   constant
}

\author[L.  Nasiri, M. Bakherad]{Leila   Nasiri$^1$ and Mojtaba Bakherad$^2$}

\address{ $^1$ Department of Mathematics  and   computer  science, Faculty  of  science,  Lorestan   University,   Khorramabad,  Iran.}
\email{leilanasiri468@gmail.com}

\address{$^2$ Department of Mathematics, Faculty of Mathematics, University of Sistan and Baluchestan, Zahedan, Iran.}
\email{mojtaba.bakherad@yahoo.com; bakherad@member.ams.org}

%~~~~~~~~~~~~~~~~~~~~~~~~~~~~~~~~~~~~~~~~~~~~~~~~~~~~~~~~~~~~~~~~~~~~~~~~~~~~~~~~~~~~~~~~~~~~~~~~~~~~~~~~~~~~~~~~~~~~~~~~~~~~~~~~~~~
\begin{abstract}
We    present  some  operator  inequalities   for  positive  linear  maps
that    generalize   and  improve   the  derived   results    in  some   recent  years.
For instant,   if  $A$ and $B$ are positive operators  and  $m,m^{'},M,M^{'}$ are positive  real  numbers
satisfying    either one of the  condition  $ 0<m \leq B \leq m^{'} <M^{'} \leq A \leq M $ or $0<m \leq A \leq m^{'} <M^{'} \leq B \leq M$,
then
     \begin{align*}
 \Phi  ^{p} \big(A \nabla _{v} B+2 r Mm (A^{-1}\nabla  B^{-1}-
&A^{-1} \sharp B^{-1}
)\big)\\
& \leq    \left( \frac{K(h)}{ 4^{\frac{2}{p}-1} K^{r_{1}} \left( \sqrt {h^{'}}\right)} \right) ^{p} \Phi^{p} (A \sharp_{\nu} B)
\end{align*}
and
\begin{align*}
 \Phi  ^{p} \big(A \nabla _{v} B+2 r Mm (A^{-1}\nabla  B^{-1}-&
A^{-1} \sharp B^{-1}
)\big) \\
&\leq    \left( \frac{K(h)}{ 4^{\frac{2}{p}-1} K^{r_{1}}\left( \sqrt {h^{'}}\right)}\right) ^{p} (\Phi(A) \sharp_{\nu}  \Phi (B))^{p},
\end{align*}
where   $\Phi$ is a  positive  unital  linear map,    $ 0 \leq \nu \leq 1$,  $p \geq  2,$ $r=\min\{\nu,1-\nu\},$   $h=\frac{M}{m},$  $h^{'}=\frac{M^{'}}{m^{'}}$,
$K(h)=\frac{(1+h)^{2}}{4h}$  and $r_{1}=\min\{2r,1-2r\}.$
 We also  obtain   a reverse  of  the  Ando   inequality
for  positive  linear  maps   via  the   Kantorovich   constant.
\end{abstract} \maketitle
%~~~~~~~~~~~~~~~~~~~~~~~~~~~~~~~~~~~~~~~~~~~~~~~~~~~~~~~~~~~~~~~~~~~~~~~~~~~~~~~~~~~~~~~~~~~~~~~~~~~~~~~~~~~~~~~~~~~~~~~~~~~~~~~~~~~
\section{Introduction and preliminaries}
Let ${\mathbb B}({\mathscr H})$ denote the $C^*$-algebra of all bounded linear operators on a complex Hilbert space ${\mathscr H}$   whose    identity   is  denoted   by $I$. An operator $A\in{\mathbb B}({\mathscr H})$ is called positive if $\langle Ax,x\rangle\geq0$ for all $x\in{\mathscr H }$  and in this case we write $A\geq0$. We write $A>0$ if $A$ is a positive invertible operator. The   absolute value   of  $A$   is  denoted  by  $|A|,$   that   is  $|A|=(A^*A)^{\frac{1}{2}}.$ For self-adjoint operators $A, B\in{\mathbb B}({\mathscr H})$, we say $A\leq B$ if $B-A\geq0$. The Gelfand map $f(t)\mapsto f(A)$ is an isometrical $*$-isomorphism between the $C^*$-algebra
$C({\rm sp}(A))$ of continuous functions on the spectrum ${\rm sp}(A)$ of a self-adjoint operator $A$ and the $C^*$-algebra generated by $A$ and  $I$. If $f, g\in C({\rm sp}(A))$, then $f(t)\geq g(t)\,\,(t\in{\rm sp}(A))$ implies that $f(A)\geq g(A)$.
A   linear  map  $\Phi$  is  positive  if  $\Phi (A) \geq 0$
whenever   $A \geq 0.$  It  is  said  to  be   unital  if  $\Phi (I)=I.  $
   If  $A,B \in {\mathbb B}({\mathscr H})$ be  positive invertible,
  then  the    $\nu-$weighted   arithmetic  mean   and  geometric   mean  of  $A$  and  $B$
denoted   by  $A\nabla_{\nu} B$   and    $A\sharp_{\nu}  B,$  respectively, which are defined
  by
  \begin{align*}
A\nabla_{\nu}B=\nu A+(1-\nu)B
\quad\textrm{and}\quad
 A\sharp_\nu B=A^{\frac{1}{2}}\left(A^{-\frac{1}{2}}BA^{-\frac{1}{2}}\right)^{\nu}A^{\frac{1}{2}},
\end{align*}
respectively, where  $0 \leq \nu \leq 1. $
In case  of  $\nu=\frac{1}{2},$  we  write  $A\nabla B$   and   the
 $A\sharp B$   for    the   arithmetic  mean   and   the   geometric  mean, respectively.
 The   well-known  $\nu-$weighted   arithmetic-geometric (AM-GM)  operator  inequality
   says   that  if
 $A,B\in{\mathbb B}({\mathscr H})$   are  positive   and  $0\leq \nu \leq 1,$  then
 $A \sharp_{\nu} B  \leq    A \nabla _{\nu} B$;   see
  \cite{TFurutaJMicicHotJPecaricYSeo}.
 For $\nu=\frac{1}{2},$ we  obtain  the  AM-GM    operator  inequality
 \begin{equation}\label{12}
 A \sharp B  \leq   \frac {A + B}{2}.
 \end{equation}
For further information about the  AM-GM  operator  inequality and positive linear maps inequalities we refer the reader to
\cite{TAndo, MB, MB1, MFujiiRNakamoto,MSMoslehianRNakamotoYSeo} and references therein.   Lin   \cite{MLin}  presented  a  reverse  of   inequality  $\eqref{12}$   for  a positive  linear  map  $\Phi$  and   positive  operators   $A,B \in {\mathbb B}({\mathscr H})$
such  that
$m \leq A,B \leq M$
as  follows:
\begin{equation}\label{13}
\Phi  \left(\frac {A + B}{2} \right)\leq K(h) \Phi(A \sharp  B),
 \end{equation}
where  $K(h)=\frac{(1+h)^{2}}{4h}$  and  $h=\frac{M}{m}.$
The  constant   $K(t)=\frac{(t+1)^{2}}{4t}(t>0)$ is
called   the  Kantorovich  constant    which   satisfies  the  following  properties:\\
$($i$)$  $K(1,2)=1;$\\
$($ii$)$  $K(t,2)=K(\frac{1}{t},2) \geq 1 \,\,(t>0);$\\
$($iii$)$  $K(t,2)$   is  monotone  increasing on   the   interval  $[1,\infty)$
and    monotone   decreasing on   the   interval   $(0,1]$.

The  Lowner-Heinz    theorem   \cite{FKuboTAndo}   says   that   if  $ A, B\in {\mathbb B}({\mathscr H})$    are    positive,  then  for  $0 \leq p \leq 1,$
\begin{equation}\label{14}
A \leq B   \quad\;\;  {\textrm{ implies}}  \quad   A^{p} \leq B^{p}.
\end{equation}
In general \eqref{14} is not true for $p > 1$.
In   \cite{MLin}, the author  showed   that  inequality  $\eqref{13}$
 can  be  squared   that  is,
\begin{equation}\label{15}
\Phi^{2} \left(\frac {A+B}{2} \right) \leq K^{2}(h)  \Phi^{2} (A\sharp B)
\end{equation}
and
\begin{equation}\label{16}
\Phi^{2} \left(\frac {A+B}{2} \right) \leq K^{2}(h)  (\Phi(A) \sharp \Phi( B))^{2}.
\end{equation}
It  follows  $(\ref{14}),$  $(\ref{15})$  and  $(\ref{16})$    that  for  $0< p \leq 2$ we have
\begin{equation}\label{17}
\Phi^{p} \left(\frac {A+B}{2} \right) \leq K^{p}(h)  \Phi^{p} (A\sharp B)
\end{equation}
and
\begin{equation}\label{18}
\Phi^{p} \left(\frac {A+B}{2} \right) \leq K^{p}(h)  (\Phi(A) \sharp \Phi( B))^{p}.
\end{equation}
It  is   natural  to  ask  whether   inequalities   (\ref{17}) and  (\ref{18})  are  true  for  $p>2.$
In  \cite{XFuCHe},    the  authors    gave   a   positive    answer     to  this   question  and  proved   the  following  theorem:
\begin{theorem}\label{hoho}
Let  $0< m \leq A,B \leq M.$  Then  for  every  positive  unital   linear  map $\Phi$   and  for every $p \geq 2$
\begin{align}\label{19}
\Phi^{p} \left(\frac {A+B}{2} \right)  & \leq  \left( \frac{(M+m)^{2}}{4^{\frac{2}{p}}Mm} \right)^{p} \Phi^{p} (A\sharp B)
\end{align}
and
\begin{align}\label{110}
 \Phi^{p} \left(\frac {A+B}{2} \right)  & \leq \left( \frac{(M+m)^{2}}{4^{\frac{2}{p}}Mm} \right)^{p} (\Phi(A) \sharp \Phi( B))^{p}.
\end{align}
\end{theorem}
The next   result  is  a     further  generalization  \cite{MB}:

\begin{theorem}\cite{MB}\label{Bha}
Let  $0< m \leq A,B \leq M.$  Then  for  every  positive  unital   linear  map $\Phi,$  $0 \leq \nu \leq 1$  and  for every $p>0$
\begin{align*}
\Phi^{p} \left( A\nabla _{\nu} B +2r Mm \left( A^{-1} \nabla B^{-1}-A^{-1} \sharp B^{-1} \right)  \right)  & \leq
\alpha^{p} \Phi^{p} (A\sharp_{\nu} B)  \\
\Phi^{p} \left( A\nabla_{\nu} B +2r Mm \left( A^{-1} \nabla B^{-1}-A^{-1} \sharp B^{-1} \right)  \right)   & \leq
\alpha^{p} (\Phi(A) \sharp_{\nu}  \Phi( B))^{p},
\end{align*}
where  $r=\min\{ \nu, 1-\nu\}$  and $\alpha=\max \left\{  \frac{(M+m)^{2}}{4 Mm},  \frac{(M+m)^{2}}{4^{\frac{2}{p}}Mm}   \right\}.$

\end{theorem}
The  authors  of  \cite{CYangDWang}   proved   the  following  theorem$,$  which  is  another
 improvement   of    inequalities   \eqref{19} and  \eqref{110}.
\begin{theorem}\cite{CYangDWang}\label{moradi}
Let
$
0<m \leq A \leq m^{'} <M^{'} \leq B \leq M. $  Then  for  every  positive  unital   linear  map $\Phi,$  $0 \leq \nu \leq 1$  and  for every $p\geq 2$
\begin{align*}
\Phi^{p} ( A\nabla  _{\nu}B )   \leq
\left(  \frac{K(h)}{4^{\frac{2}{p}-1} K^{r}(h^{'})}  \right)^{p}
 \Phi^{p} (A\sharp_{\nu} B)  \\
\Phi^{p} ( A\nabla  _{\nu}B )   \leq
\left(  \frac{K(h)}{4^{\frac{2}{p}-1} K^{r}(h^{'})}  \right)^{p}
 (\Phi(A) \sharp_{\nu}  \Phi( B))^{p},
\end{align*}
where  $r=\min\{ \nu, 1-\nu\},$   $h=\frac{M}{m},$    $h^{'}=\frac{M^{'}}{m^{'}}$ and  $K(h)=\frac{(1+h)^{2}}{4h}$.
\end{theorem}
In  this  article,  we   give   some  operator  inequalities  involving positive  linear  maps
that   generalize    inequalities  \eqref{19},   \eqref{110}  and  refine   some  results    in \cite{MB, CYangDWang}.
 Moreover,   we   obtain a  reverse    of     Ando's  inequality.

%#################################################################################################################################################################

\section{Some  operator  inequalities  involving positive  linear  maps}
We begin this section with several  essential  lemmas.
\begin{lemma}\cite{RB1} $($Choi's  inequality$)$
Let $A \in{\mathbb B}({\mathscr H})$ be positive and $\Phi $  be  a   positive  unital   linear  map.   Then
\begin{align}\label{23}
\Phi (A)^{-1} \leq \Phi  \left(A^{-1} \right).
\end{align}
\end{lemma}
The next lemma, part $\rm{(i)}$ is proved for matrices but a careful investigation shows that it is true
for operators on an arbitrary Hilbert space; see \cite[page 79]{moslehian2}.
\begin{lemma}\cite{RBhatiaFKittaneh,FKuboTAndo, MB}\label{6}
Let $A, B\in{\mathbb B}({\mathscr H})$ be positive   and $\alpha>0$. Then \\
$\rm{(i)}\,\,||{\rm AB}||\leq\frac{1}{4}||A+B||^{2}.$\\
$\rm{(ii)}\,\,||A^{\alpha}+B^{\alpha}||\leq ||(A+B)^{\alpha}||.$\\
$\rm{(iii)}\,\, A\leq\alpha B$ if and only if $||A^{\frac{1}{2}}B^{-\frac{1}{2}}||\leq \alpha^{\frac{1}{2}}.$
\end{lemma}
To  obtain  our  results,     we   need  to    prove  the  following  lemma.    Its   proof   is   similar to that of   \cite[Theorem 3.1]{JLWuJGZhao}.
\begin{lemma}\label{233}
Suppose  that $A,B\in {\mathbb B}({\mathscr H})$ are positive   and  $m,m^{'},M,M^{'}$ are positive  real  numbers
satisfying    either one of   the  following  conditions:
\begin{enumerate}
\item
 $ 0<m \leq B \leq m^{'} <M^{'} \leq A \leq M $;
\item
$
0<m \leq A \leq m^{'} <M^{'} \leq B \leq M. $
\end{enumerate}
Then  for   every  $ 0 \leq \nu \leq 1$,
\begin{align}\label{25}
2r \left( A^{-1} \nabla  B^{-1} -A^{-1} \sharp  B^{-1}\right) +K^{r_{1}} \left( \sqrt {h^{'}}\right) \left( A^{-1} \sharp_{\nu}  B^{-1}\right)
\leq \left( A^{-1}  \nabla_{\nu} B^{-1} \right),
\end{align}
   where
$r=\min\{\nu,1-\nu\},$   $r_{1}=\min\{2r,1-2r\},$     $K(h)=\frac{(1+h)^{2}}{4h},$   $h=\frac{M}{m}$  and
$h^{'}=\frac{M^{'}}{m^{'}}.$
\end{lemma}
\begin{proof}
It follows  from \cite[Lemma 2.3]{JLWuJGZhao} that
\begin{align*}
2r \left( \frac {1+x}{2}-\sqrt{x}\right)+K^{r_{1}}(\sqrt{x}) x^{\nu} \leq (1-\nu)+\nu x
\end{align*}
for  any  $x>0.$   The  first   condition,  that is,
$
0<m \leq B \leq m^{'} <M^{'} \leq A \leq M
$
ensures   that
$
1< h^{'} I =\frac{M^{'}}{m^{'}}I \leq A^{\frac{1}{2}} B^{-1} A^{\frac{1}{2}} \leq  \frac{M}{m}I =h I.
$
By  setting  $X=A^{\frac{1}{2}} B^{-1} A^{\frac{1}{2}},$  we  see
$  {\rm  sp}(X)\subseteq  [h^{'},h] \subset (1,+\infty).$  Now,  the monotonicity  principle    for  operator  functions   yields   the   inequality
\begin{align}\label{26}
 (1-\nu)+\nu  X  &\geq   2r \left( \frac {I+X}{2}-\sqrt{X}\right)+ \min_{h^{'} \leq x \leq h}K^{r_{1}}(\sqrt{x}) X^{\nu}  \nonumber \\
 &\geq  2r \left( \frac {I+X}{2}-\sqrt{X}\right)+ K^{r_{1}}\left(\sqrt{h^{'}} \right) X^{\nu}.
\end{align}
The   last  above inequality    follows   by  the increasing   property  of  the  function
$K(t)$   on   the  interval  $(1,+\infty) $;  see    \cite{TFurutaJMicicHotJPecaricYSeo}.
 Finally,   multiplying  the both  sides  of    inequality \eqref{26}  by  $A^{-\frac{1}{2}},$   we obtain
the  desired  result.
 The   inequality   can   be  proved    under   the   second  condition �� (2) in  a  similar   way.

\end{proof}
Our   first  main  result  is  the  following:
\begin{theorem}\label{nasi-main1}
Suppose that $A,B\in {\mathbb B}({\mathscr H})$ are positive   and  $m,m^{'},M,M^{'}$ are positive  real  numbers
satisfying    either  one of   the  following  conditions:
\begin{enumerate}
\item
 $ 0<m \leq B \leq m^{'} <M^{'} \leq A \leq M $;
\item
$
0<m \leq A \leq m^{'} <M^{'} \leq B \leq M. $
\end{enumerate}
Then   for every  positive  unital   linear map $\Phi $   and  every  $ 0 \leq \nu \leq 1$
\begin{align}\label{27}
 \Phi  ^{2} \big(A \nabla _{v} B+2 r Mm \big(A^{-1}\nabla  B^{-1} &-
A^{-1} \sharp B^{-1}
\big)\big) \nonumber\\
&\leq    \left( \frac{K(h)}{ K^{r_{1}}\left( \sqrt {h^{'}} \right)} \right) ^{2} \Phi^{2} (A \sharp_{\nu} B)
\end{align}
and
\begin{align}\label{281}
 \Phi  ^{2}\big(A \nabla _{v} B+2 r Mm \big(A^{-1}\nabla  B^{-1}&-
A^{-1} \sharp B^{-1}
\big)\big) \nonumber\\
&\leq   \left( \frac{K(h)}{ K^{r_{1}}\left( \sqrt {h^{'}} \right)} \right) ^{2} ( \Phi(A ) \sharp_{\nu}\Phi(B))^{2},
\end{align}
where $r=\min\{\nu,1-\nu\},$  $K(h)=\frac{(1+h)^{2}}{4h},$
 $h=\frac{M}{m},$  $h^{'}=\frac{M^{'}}{m^{'}}$  and
$r_{1}=\min\{2r,1-2r\}.$
\end{theorem}
\begin{proof}
We  shall  prove     inequality  \eqref{27},   and     leave    inequality \eqref{281}  to  the   reader   because   the  proof  is   similar.
By  Lemma \ref{233},   inequality \eqref{27} is equivalent to
\begin{align*}
& \left\| \Phi \left(A \nabla _{v} B+2 r Mm \left(A^{-1}\nabla  B^{-1}-
A^{-1} \sharp B^{-1}
\right)\right) Mm K^{r_{1}}\left( \sqrt {h^{'}}\right)  \Phi^{-1} (A \sharp_{\nu} B)  \right\| \\
&  \leq  \frac{(M+m)^{2}}{4}.
\end{align*}
Using    Lemma \ref{6}, inequalities  \eqref{23},   \eqref{25} and   the linear  property  of $\Phi,$   we obtain
\begin{align*}
& \Big\|  \Phi \left(A \nabla _{v} B+2 r Mm  \left(A^{-1}\nabla  B^{-1}-
A^{-1} \sharp B^{-1}
\right)\right)
 Mm K^{r_{1}}\left( \sqrt {h^{'}}\right) \Phi^{-1}(A\sharp_{\nu} B) \Big \| \\
 & \leq
\frac{1}{4}
 \Big\| \Phi \Big(A \nabla _{v} B+2 r Mm \Big(A^{-1}\nabla  B^{-1}-
A^{-1} \sharp B^{-1}
\Big)\Big) \\
&\,\,\, + Mm K^{r_{1}}\left( \sqrt {h^{'}} \right) \Phi \Big(A^{-1}\sharp_{\nu} B^{-1}\Big)  \Big\|^{2}\\
&=
\frac{1}{4}
\bigg\| \Phi (A \nabla _{v} B)+ Mm \Big(\Phi \Big ( 2r  \Big(A^{-1}\nabla  B^{-1}- A^{-1} \sharp B^{-1} \Big) \\
&\,\,\,+ K^{r_{1}}\left( \sqrt {h^{'}} \right) \Big(A^{-1}\sharp_{\nu} B^{-1} \Big) \Big) \Big) \bigg\|^{2} \\
& \leq  \frac{1}{4}
\Big\| \Phi \left(A \nabla _{v} B \right) +Mm   \Phi \Big(A^{-1} \nabla _{v} B^{-1} \Big) \Big\|^{2}  \\
& \leq  \frac{(M+m)^{2}}{4}.
\end{align*}
The  last above inequality     holds
since   by    our  assumptions,
\begin{align*}
A+Mm A^{-1} \leq Mm
\qquad \textrm{and}\qquad
B+Mm B^{-1} \leq Mm.
\end{align*}
By  multiplying   the  inequalities  above   by  $(1-\nu)$  and  $\nu,$  respectively,   and  then  summing up  the  derived  inequalities,   we  get
$$  A \nabla _{v} B +Mm  \Big(A^{-1} \nabla _{v} B^{-1} \Big)
 \leq  M+m. $$
Since   $\Phi$   is  a  positive  linear map,  we  obtain
  $$ \Phi  \left(A \nabla _{v} B \right) +Mm   \Phi \Big(A^{-1} \nabla _{v} B^{-1} \Big)
 \leq  M+m. $$
So,  inequality \eqref{27}   holds. \\
\end{proof}
\begin{corollary}
Suppose that $A,B\in {\mathbb B}({\mathscr H})$ are positive   and $m,m^{'},M,M^{'}$ are positive  real  numbers
satisfying    either  one of   the  following  conditions:
\begin{enumerate}
\item
 $ 0<m \leq B \leq m^{'} <M^{'} \leq A \leq M $;
\item
$
0<m \leq A \leq m^{'} <M^{'} \leq B \leq M. $
\end{enumerate}
Then
   for every   positive  unital   linear map $\Phi ,$   $ 0 \leq \nu \leq 1$  and  for  every  $ 0 < p \leq 2$
\begin{align*}
 \Phi  ^{p} \big(A \nabla _{v} B+2 r Mm \big(A^{-1}\nabla  B^{-1}&-
A^{-1} \sharp B^{-1}
\big)\big)\nonumber \\
& \leq    \left( \frac{K(h)}{ K^{r_{1}} \left( \sqrt {h^{'}} \right)} \right) ^{p} \Phi^{p} (A \sharp_{\nu} B)
\end{align*}
and
\begin{align*}
 \Phi  ^{p} \big(A \nabla _{v} B+2 r Mm \big(A^{-1}\nabla  B^{-1}&-
A^{-1} \sharp B^{-1}
\big)\big)\nonumber \\
& \leq   \left( \frac{K(h)}{ K^{r_{1}} \left( \sqrt {h^{'}} \right)} \right) ^{p} (\Phi (A) \sharp_{\nu} \Phi (B))^{p},
\end{align*}
where $r=\min\{\nu,1-\nu\},$  $K(h)=\frac{(1+h)^{2}}{4h},$
  $h=\frac{M}{m},$  $h^{'}=\frac{M^{'}}{m^{'}}$  and
$r_{1}=\min\{2r,1-2r\}.$
\end{corollary}
\begin{proof}
If  $ 0< p\leq 2,$  then  $0 < \frac{p}{2} \leq 1$. Using    inequalities   \eqref{14},  \eqref{27}  and   \eqref{281},    we   get
the   desired results.
\end{proof}

\begin{theorem}\label{nasi-main123}
Suppose  that $A,B\in{\mathbb B}({\mathscr H})$ are positive   and  $m,m^{'},M,M^{'}$ are positive  real  numbers
satisfying    either one of   the  following  conditions:
\begin{enumerate}
\item
 $ 0<m \leq B \leq m^{'} <M^{'} \leq A \leq M $;
\item
$
0<m \leq A \leq m^{'} <M^{'} \leq B \leq M. $
\end{enumerate}
Then
   for every  positive  unital   linear map $\Phi ,$   $ 0 \leq \nu \leq 1$  and  for  every  $p \geq  2,$  we  have
   \begin{align}\label{211}
 \Phi  ^{p} \big(A \nabla _{v} B+2 r Mm (A^{-1}\nabla  B^{-1}&-
A^{-1} \sharp B^{-1}
)\big)\nonumber \\
& \leq    \left( \frac{K(h)}{ 4^{\frac{2}{p}-1} K^{r_{1}} \left( \sqrt {h^{'}}\right)} \right) ^{p} \Phi^{p} (A \sharp_{\nu} B)
\end{align}
and
\begin{align}\label{212}
 \Phi  ^{p} \big(A \nabla _{v} B+2 r Mm (A^{-1}\nabla  B^{-1}&-
A^{-1} \sharp B^{-1}
)\big) \nonumber\\
&\leq    \left( \frac{K(h)}{ 4^{\frac{2}{p}-1} K^{r_{1}}\left( \sqrt {h^{'}}\right)}\right) ^{p} (\Phi(A) \sharp_{\nu}  \Phi (B))^{p},
\end{align}
where $r=\min\{\nu,1-\nu\},$  $K(h)=\frac{(1+h)^{2}}{4h},$
$K(h^{'})=\frac{(1+h^{'})^{2}}{4 h^{'}},$  $h=\frac{M}{m},$  $h^{'}=\frac{M^{'}}{m^{'}}$  and
$r_{1}=\min\{2r,1-2r\}.$
   \end{theorem}
\begin{proof}
Since the  proof  of  inequality  \eqref{212}   is  similar  to  the  proof  of   inequality  \eqref{211},   we  only   prove   inequality
 (\ref{211}).
By Lemma \ref{6},     inequality   \eqref{211} is   equivalent  to
{\footnotesize\begin{eqnarray*}
\left\| \Phi ^{\frac{p}{2}} \left(A \nabla _{v} B+2 r Mm \left(A^{-1}\nabla  B^{-1}-
A^{-1} \sharp B^{-1}
\right)\right) \Phi^{-\frac{p}{2}} (A \sharp_{\nu} B)  \right\| \leq
\left( \frac{K(h)}{ 4^{\frac{2}{p}-1} K^{r_{1}} \left( \sqrt {h^{'}}\right)} \right) ^{\frac{p}{2}} .
\end{eqnarray*}}
Using Lemma \ref{6},  inequalities     \eqref{23},  \eqref{25},   and   applying   the   same   reasoning   as   in   the  last   inequality  of   Theorem \ref{nasi-main1},
 we  have
{\footnotesize\begin{align*}
& M ^{\frac{p}{2}}m^{\frac{p}{2}} \Big\| \Phi^{\frac{p}{2}} \left(A \nabla _{v} B+2r Mm \left(A^{-1}\nabla  B^{-1}-
A^{-1} \sharp B^{-1}
\right) \right) K^{\frac{r_{1}p}{2}} \left(\sqrt {h^{'}} \right)
\Phi^{-\frac{p}{2}}(A\sharp_{\nu} B)  \Big\| \\
&=
 \Big\| \Phi^{\frac{p}{2}} \left(A \nabla _{v} B+2rMm \left(A^{-1}\nabla  B^{-1}-
A^{-1} \sharp B^{-1}
\right)\right) M^{\frac{p}{2}}m^{\frac{p}{2}} K ^{\frac {pr_{1}}{2}}\left(\sqrt {h^{'}} \right)\Phi^{-\frac{p}{2}}(A\sharp_{\nu} B)  \Big\| \\
& \leq  \frac{1}{4}
 \Big\|  \Phi^{\frac{p}{2}} \left(A \nabla _{v} B+2rMm \left(A^{-1}\nabla  B^{-1}-
A^{-1} \sharp B^{-1}
\right) \right) \\
& \quad +  M^{\frac{p}{2}}m^{\frac{p}{2}}  K ^{\frac {pr_{1}}{2}}\left(\sqrt {h^{'}} \right) \Phi^{-\frac{p}{2}}(A\sharp_{\nu} B) \Big\|^{2} \\
& \leq  \frac{1}{4}
 \Big\|  \Big( \Phi \left(A \nabla _{v} B+2r Mm \left(A^{-1}\nabla  B^{-1}-A^{-1} \sharp  B^{-1}\right)   \right) \\
& \quad + Mm K ^{r_{1}} \left(\sqrt {h^{'}} \right)\Phi^{-1}(A\sharp_{\nu} B)  \Big)^{\frac{p}{2}} \Big \|^{2}   \\
&=
\frac{1}{4}
 \left\|   \Phi \left(A \nabla _{v} B+2r Mm \left(A^{-1}\nabla  B^{-1}-A^{-1} \sharp  B^{-1}\right)   \right)
+ Mm K ^{r_{1}} \left(\sqrt {h^{'}} \right)\Phi^{-1}(A\sharp_{\nu} B)   \right\|^{p}   \\
& \leq
\frac{1}{4}
 \left\|   \Phi \left(A \nabla _{v} B+2r Mm \left(A^{-1}\nabla  B^{-1}-A^{-1} \sharp  B^{-1}\right)   \right)
+ Mm K ^{r_{1}} \left(\sqrt {h^{'}} \right)\Phi(A^{-1}\sharp_{\nu} B^{-1})   \right\|^{p}   \\
& =
\frac{1}{4}
 \Big\| \Phi \left(A \nabla _{v} B \right)+Mm  \Big(\Phi \Big(2r \Big(A^{-1}\nabla  B^{-1}-A^{-1} \sharp  B^{-1}\Big)
+ K ^{r_{1}} \left(\sqrt {h^{'}} \right) \Big(A^{-1}\sharp_{\nu} B^{-1} \Big) \Big) \Big)  \Big\|^{p}   \\
& \leq
 \frac{1}{4}
\left \|  \Phi (A \nabla _{v} B)+Mm \Phi \Big(A^{-1}\nabla  B^{-1}
\Big)
 \right\|^{p} \\
&\leq   \frac{1}{4} (M+m)^{p}.
\end{align*}}
Thus we get the desired result.
\end{proof}
\begin{remark}
For  $p \geq 1,$   we  have
$$\Phi^{p}(A \nabla _{v} B) \leq \Phi^{p}(A \nabla _{v} B) +( 2r Mm )^{p}  \Phi^{p}\Big(A^{-1}\nabla  B^{-1}-A^{-1} \sharp  B^{-1}\Big). $$
On  the  other  hand,  Lemma \ref{6}   yields that
\begin{align*}
 \left \|\Phi^{p}(A \nabla _{v} B) \right\|
& \leq \Big\|\Phi^{p}(A \nabla _{v} B) +( 2r Mm )^{p}  \Phi^{p}\Big(A^{-1}\nabla  B^{-1}-A^{-1} \sharp  B^{-1}\Big) \Big \| \\
& \leq \Big\| \Phi^{p}\Big(A \nabla _{v} B +2r Mm \Big(A^{-1}\nabla  B^{-1}-A^{-1} \sharp  B^{-1}\Big) \Big) \Big\|.
\end{align*}
Therefore,   Theorem \ref{nasi-main123}  is  a  refinement  of   Theorem \ref{moradi}   for   the  operator  norm and  $p \geq 2.$
\end{remark}
\begin{remark}
Since  the  Kantorovich  constant   $K(h)$  is   an   increasing   function  on  the  interval
$(1, +\infty )$ and  also $K(h) \geq 1$    for   every  $h >0,$ so
Theorem \ref{nasi-main123}  is  a  refinement  of   Theorem \ref{Bha}; see \cite{TFurutaJMicicHotJPecaricYSeo}.
\end{remark}
Zhang   \cite{PZhang}   obtained    the  following   inequalities    for  $p \geq 4:$
\begin{align*}
\Phi^{p} ( A\nabla  B )   \leq
\left(  \frac{K(h) \left(M^{2}+m^{2} \right)}{4^{\frac{2}{p}}Mm }  \right)^{p}
 \Phi^{p} (A\sharp  B);  \\
\Phi^{p} ( A\nabla  B )   \leq
\left(  \frac{K(h) \left(M^{2}+m^{2} \right)}{4^{\frac{2}{p}}Mm }  \right)^{p}
 (\Phi(A) \sharp \Phi( B))^{p}.
\end{align*}
Recently,      the  authors  of   \cite{CYangDWang}   improved   the  above inequalities      as  follows:
\begin{align}\label{28}
\Phi^{p} ( A\nabla_{\nu}  B )   \leq
\left(  \frac{K(h) \left(M^{2}+m^{2} \right)}{4^{\frac{2}{p}}Mm K^{r}(h^{'})}  \right)^{p}
 \Phi^{p} (A\sharp_{\nu}  B);
 \end{align}
 \begin{align}\label{29}
\Phi^{p} ( A\nabla_{\nu}  B )   \leq
\left(  \frac{K(h) \left(M^{2}+m^{2} \right)}{4^{\frac{2}{p}}Mm K^{r}(h^{'})}  \right)^{p}
 (\Phi(A) \sharp_{\nu} \Phi( B))^{p}.
\end{align}
In the following theorem, we show some refinements of inequalities \eqref{28} and \eqref{29}.
\begin{theorem}\label{nasi-main1234}
Let   $A,B\in {\mathbb B}({\mathscr H})$ are positive   and $m,m^{'},M,M^{'}$ are positive  real  numbers
satisfying     either  one of   the  following  conditions:
 \begin{enumerate}
\item
 $ 0<m \leq B \leq m^{'} <M^{'} \leq A \leq M $;
\item
$
0<m \leq A \leq m^{'} <M^{'} \leq B \leq M. $
\end{enumerate}
Then
   for every  positive  unital    linear map $\Phi,$   $0 \leq \nu \leq 1$  and  for  every  $p \geq  4$
   \begin{align}\label{213}
 \Phi ^{p}   \big(A \nabla _{v} B+2 r Mm (A^{-1}\nabla  B^{-1}&-
A^{-1} \sharp B^{-1}
)\big) \nonumber\\
& \leq  \left(  \frac{K(h) \left(M^{2}+m^{2} \right)}{4^{\frac{2}{p}}Mm K^{r}(h^{'})}  \right)^{p}
 \Phi^{p} (A \sharp_{\nu} B)
 \end{align}
and
 \begin{align}\label{214}
 \Phi^{p}   \big(A \nabla _{v} B+2 r Mm (A^{-1}\nabla  B^{-1}&-
A^{-1} \sharp B^{-1}
)\big) \\
& \leq
\left(  \frac{K(h) \left(M^{2}+m^{2} \right)}{4^{\frac{2}{p}}Mm K^{r}(h^{'})}  \right)^{p}
 (\Phi (A) \sharp_{\nu}  \Phi (B))^{p},  \nonumber
\end{align}
where $r=\min\{\nu,1-\nu\},$  $K(h)=\frac{(1+h)^{2}}{4h},$
  $h=\frac{M}{m},$  $h^{'}=\frac{M^{'}}{m^{'}}$  and
$r_{1}=\min\{2r,1-2r\}.$
   \end{theorem}
\begin{proof}
It   follows   from    Lemma \ref{6}  and     Theorem \ref{nasi-main1}   that
\begin{eqnarray*}
&& M ^{\frac{p}{2}}m^{\frac{p}{2}} \left\| \Phi^{\frac{p}{2}} \left(A \nabla _{v} B+2r Mm \left(A^{-1}\nabla  B^{-1}-
A^{-1} \sharp B^{-1}
\right) \right)
\Phi^{-\frac{p}{2}}(A\sharp_{\nu} B)  \right\| \\
&=&
 \left\| \Phi^{\frac{p}{2}} \left(A \nabla _{v} B+2rMm \left(A^{-1}\nabla  B^{-1}-
A^{-1} \sharp B^{-1}
\right)\right) M^{\frac{p}{2}}m^{\frac{p}{2}} \Phi^{-\frac{p}{2}}(A\sharp_{\nu} B)  \right\|
 \\&\leq &  \frac{1}{4}
 \Big\|   \frac{ K^{\frac{r_{1}p}{4}}\left(\sqrt{h^{'}}\right) }
{K^{\frac{p}{4}}(h)}
 \Phi^{\frac{p}{2}} \left(A \nabla _{v} B+2rMm \left(A^{-1}\nabla  B^{-1}-
A^{-1} \sharp B^{-1}
\right) \right)    \\
 &&+ \left(  \frac{  M^{2}m^{2}   K(h)}{K^{r_{1}}\left( \sqrt {h^{'}}\right)}  \right)^{\frac{p}{4}}
  \Phi^{-\frac{p}{2}}(A\sharp_{\nu} B)  \Big\|^{2} \\
   &\leq& \frac{1}{4}
 \Big\|  \Big( \frac{K^{r_{1}}\left( \sqrt {h^{'}}\right)}{K(h)}
 \Phi^{2} \left(A \nabla _{v} B+2rMm \left(A^{-1}\nabla  B^{-1}-
A^{-1} \sharp B^{-1}
\right) \right)   \\
 &&+\frac{M^{2}m^{2}K(h)}{K^{r_{1}}\left( \sqrt {h^{'}}\right)}
  \Phi^{-2}(A\sharp_{\nu} B)  \Big)^{\frac{p}{4}}  \Big\|^{2} \\
& =& \frac{1}{4}
 \Big\|  \frac{K^{r_{1}}\left( \sqrt {h^{'}}\right)}{K(h)}
 \Phi^{2} \left(A \nabla _{v} B+2rMm \left(A^{-1}\nabla  B^{-1}-
A^{-1} \sharp B^{-1}
\right) \right)   \\
 &&+\frac{  M^{2}m^{2}  K(h)}{ K^{r_{1}}\left( \sqrt {h^{'}}\right)}
  \Phi^{-2}(A\sharp_{\nu} B)  \Big\|^{\frac{p}{2}} \\
   & \leq &   \frac{1}{4}
 \Big\|  \frac{K(h)}{K^{r_{1}}\left( \sqrt {h^{'}}\right)}  \Big(
 \Phi^{2} \left(A \sharp_{v} B \right) + M^{2}m^{2}
   \Phi^{-2} \left(A \sharp _{v} B
 \right)   \Big) \Big\|^{\frac{p}{2}}
 \\
& \leq&
 \frac{1}{4}  \left( \frac{K(h) (M^{2}+m^{2})}{K^{r_{1}}\left( \sqrt {h^{'}}\right)} \right)^{\frac{p}{2}}.
 \end{eqnarray*}

It follows from
$0<m\leq A \sharp _{\nu}B  \leq M$
and  the linearity  $\Phi$   that
$0<m\leq \Phi \left( A \sharp _{\nu}B \right) \leq M.$
In    addition,    for  every  $T\in {\mathbb B}({\mathscr H})$  such  that  $0<m\leq T \leq M,$  we  have
$$M^{2}m^{2} T^{-2}+T^{2} \leq M^{2}+m^{2}.$$
Now,  by  putting  $\Phi \left( A \sharp _{\nu}B \right) $
in  the  latter  inequality,  we   obtain   the  last  inequality.
Hence,
\begin{align*}
 M ^{\frac{p}{2}}m^{\frac{p}{2}}  \Big\|  \Phi^{\frac{p}{2}} \big(A \nabla _{v} B+2r Mm \big(A^{-1}\nabla  B^{-1}\-
A^{-1} \sharp B^{-1}
\big) \big)&
\Phi^{-\frac{p}{2}}(A\sharp_{\nu} B)  \Big\| \\
& \leq   \left(  \frac{K(h) \left(M^{2}+m^{2} \right)}{4^{\frac{2}{p}}Mm K^{r}(h^{'})}  \right)^{\frac{p}{2}}.
\end{align*}
By Lemma \ref{6},  the last  inequality    implies    inequality \eqref{213}.
Analogously,    we can prove inequality \eqref{214}.
\end{proof}
\begin{remark}
Note  that     inequalities   \eqref{213}    and  \eqref{214}  are   refinements    of   \eqref{28}    and  \eqref{29}
for  the  operator   norm,  respectively.
\end{remark}
\begin{theorem}\label{nasi-main12345}
Let  $A,B\in {\mathbb B}({\mathscr H})$  are positive   and  $m,m^{'},M,M^{'}$ positive  real  numbers
satisfying    either one of   the  following  conditions:
 \begin{enumerate}
\item
 $ 0<m \leq B \leq m^{'} <M^{'} \leq A \leq M $.

\item
$
0<m \leq A \leq m^{'} <M^{'} \leq B \leq M. $
\end{enumerate}
Then
   for every   positive  unital  linear map $\Phi $   and   $ 0 \leq \nu \leq 1$
   \begin{align}\label{215}
 \Phi  ^{p} \big(A \nabla _{v} B+2 r Mm &\left(A^{-1}\nabla  B^{-1}-
A^{-1} \sharp B^{-1}
\right)\big) \nonumber \\
&\leq   \frac{  \left( K^{-\frac{r_{1}\alpha}{2}} \left(\sqrt {h^{'}}\right) K^{\frac{\alpha}{2}}(h) (M^{\alpha}+m^{\alpha}) \right)^{\frac{2 p}{\alpha}}}{ 16 M^{p}m^{p}} \Phi^{p} (A \sharp_{\nu} B),
\end{align}
\begin{align}\label{216}
 \Phi  ^{p} \big(A \nabla _{v} B+2 r Mm &\left(A^{-1}\nabla  B^{-1}-
A^{-1} \sharp B^{-1}
\right)\big) \nonumber\\
&\leq   \frac{  \left( K^{-\frac{r_{1}\alpha}{2}}\left(\sqrt {h^{'}}\right)K^{\frac{\alpha}{2}}(h) (M^{\alpha}+m^{\alpha}) \right)^{\frac{2 p}{\alpha}}}{ 16 M^{p}m^{p}} (\Phi(A) \sharp_{\nu}\Phi( B))^{p},
\end{align}
where  $1 \leq  \alpha \leq 2$, $K(h)=\frac{(1+h)^{2}}{4h},$  and  $p \geq 2 \alpha.$
   \end{theorem}
\begin{proof}
  By  Lemma \ref{6},    inequality   \eqref{215}   is  equivalent  to  the  following  inequality
    \begin{align*}
 \Big\|  \Phi ^{\frac{p}{2}}    \big(A \nabla _{v} B+2 r Mm \big(A^{-1}\nabla  B^{-1}&-
A^{-1} \sharp B^{-1}
\big)\big)  \Phi^{-\frac{p}{2}} (A \sharp_{\nu} B) \Big\| \\
& \leq   \frac{  \left( K^{-\frac{r_{1}\alpha}{2}} \left(\sqrt {h^{'}}\right) K^{\frac{\alpha}{2}}(h) (M^{\alpha}+m^{\alpha}) \right)^{\frac{p}{\alpha}}}
{ 4 M^{\frac{p}{2}}m^{\frac{p}{2}}}.
\end{align*}
By   using  Lemma \ref{6} and  Theorem \ref{nasi-main1},
 one   can  obtain
\begin{eqnarray*}
&& M ^{\frac{p}{2}}m^{\frac{p}{2}} \left\| \Phi^{\frac{p}{2}} \left(A \nabla _{v} B+2r Mm \left(A^{-1}\nabla  B^{-1}-
A^{-1} \sharp B^{-1}
\right) \right)
\Phi^{-\frac{p}{2}}(A\sharp_{\nu} B)  \right\| \\
&=&
 \left\| \Phi^{\frac{p}{2}} \left(A \nabla _{v} B+2rMm \left(A^{-1}\nabla  B^{-1}-
A^{-1} \sharp B^{-1}
\right)\right) M^{\frac{p}{2}}m^{\frac{p}{2}} \Phi^{-\frac{p}{2}}(A\sharp_{\nu} B)  \right\|
 \\&\leq&  \frac{1}{4}
 \Big\|   \frac{K^{\frac{r_{1}p}{4}} \left( \sqrt {h^{'}}\right)}{  K^{\frac{p}{4}}(h) }
 \Phi^{\frac{p}{2}} \left(A \nabla _{v} B+2rMm \left(A^{-1}\nabla  B^{-1}-
A^{-1} \sharp B^{-1}
\right) \right)    \\
 &&+ \left(  \frac{  M^{2}m^{2}  K(h) }{ K^{r_{1}} \left( \sqrt{h^{'}}\right)}  \right)^{\frac{p}{4}}
  \Phi^{-\frac{p}{2}}(A\sharp_{\nu} B)  \Big\|^{2} \\
   &\leq & \frac{1}{4}
 \Big\|  \Big( \frac{ K^{  \frac {r_{1} \alpha}{2}}\left( \sqrt {h^{'}}\right) } {K^{\frac{\alpha}{2}}(h)}
 \Phi^{\alpha} \left(A \nabla _{v} B+2rMm \left(A^{-1}\nabla  B^{-1}-
A^{-1} \sharp B^{-1}
\right) \right)   \\
 &&+\frac{M^{\alpha}m^{\alpha} K^{\frac{\alpha}{2}}(h)}{ K^{  \frac {r_{1} \alpha}{2}}\left( \sqrt {h^{'}}\right)
 }
  \Phi^{-\alpha}(A\sharp_{\nu} B)  \Big)^{\frac{p}{2 \alpha}}  \Big\|^{2} \\
   & = & \frac{1}{4}
 \Big\|  \frac{ K^{  \frac {r_{1} \alpha}{2}}\left( \sqrt {h^{'}}\right) } {K^{\frac{\alpha}{2}}(h)}
 \Phi^{\alpha} \left(A \nabla _{v} B+2rMm \left(A^{-1}\nabla  B^{-1}-
A^{-1} \sharp B^{-1}
\right) \right)   \\
 &&+\frac{M^{\alpha}m^{\alpha} K^{\frac{\alpha}{2}}(h)}{ K^{  \frac {r_{1} \alpha}{2}}\left( \sqrt {h^{'}}\right)
 }
  \Phi^{-\alpha}(A\sharp_{\nu} B) \Big\|^{\frac{p}{\alpha}} \\
  & \leq  & \frac{1}{4}
 \Big\| \frac{K^{\frac{\alpha}{2}}(h)} { K^{  \frac {r_{1} \alpha}{2}}\left( \sqrt {h^{'}}\right) } \left(
 \Phi^{\alpha} \left(A \sharp _{v} B
 \right)
+ M^{\alpha}m^{\alpha}
 \Phi^{-\alpha} \left(A \sharp _{v} B
 \right) \right)
  \Big\|^{\frac{p}{\alpha}} \\
& \leq &
 \frac{1}{4}    \left(\frac{K^{\frac{\alpha}{2}}(h) (M^{\alpha}+m^{\alpha})}{K^{\frac{r_{1} \alpha}{2}}\left( \sqrt {h^{'}}\right)} \right)^{\frac{p}{\alpha}}.
\end{eqnarray*}
\\
By the property of the
arithmetic mean (see \cite{TFurutaJMicicHotJPecaricYSeo})
$$m=m\sharp _{\nu}m\leq A \sharp _{\nu}B \leq M\sharp _{\nu} M=M. $$

%
%As  it  is  known   to  all,
%$0<m\leq A \sharp _{\nu}B \leq M.$
Since  $\Phi$ is  linear,   we  have
$$0<m\leq \Phi \left( A \sharp_{\nu}B \right) \leq M.$$
On  the  other  hand,  for  every  $T\in {\mathbb B}({\mathscr H})$  such  that  $0<m\leq T \leq M,$  we  have
$0< T^\alpha -m^\alpha$ and $0<T^{-\alpha}-M^{-\alpha}$, whence $0< (T^\alpha -m^\alpha)(T^{-\alpha}-M^{-\alpha})$ or equivalently
$$M^{\alpha}m^{\alpha} T^{-\alpha}+T^{\alpha} \leq M^{\alpha}+m^{\alpha}.$$
Now,  by  setting  $\Phi \left( A \sharp _{\nu}B \right) $
in  the  latter  inequality   we  obtain   the  last  inequality.
This  proves    inequality  \eqref{215}.   By  utilizing    the  same  ideas  as  in   the   proof   of   inequality  \eqref{215},   we   can  reach    inequality \eqref{216}.
\end{proof}
\begin{remark}
If we take  $\alpha=1,2$, then  Theorem  \ref{nasi-main12345}   reduces  to  Theorem  \ref{nasi-main1234}   and  Theorem \ref{nasi-main123},  respectively.
%This  shows  that  Theorem \ref{nasi-main12345} generalizes    Theorem \ref{nasi-main1234}   and  Theorem \ref{nasi-main123}.
\end{remark}

\section{Reverse   of  Ando's    inequality}

For  positive operators $A,B\in  {\mathbb B}({\mathscr H}),$  we  know  \cite{RB1}  that
for  every  positive  unital  linear  map $\Phi$
\begin{eqnarray}\label{217}
\Phi(A \sharp B) \leq \Phi(A) \sharp \Phi (B).
\end{eqnarray}
Ando's   inequality   says   that   if    $A,B$  be  positive operators    and  $\Phi$    be  a
 positive  unital  linear  map,  then
\begin{eqnarray}\label{2177}
\Phi(A \sharp_{\nu} B) \leq \Phi(A) \sharp_{\nu} \Phi (B).
\end{eqnarray}

The  author  \cite{EYLee} presented   the  following  theorem  that  can  be  viewed  as  a  reversed   version  of   (\ref{217}).
\begin{theorem}
If  $0< m^{2}_{1}\leq A \leq M^{2}_{1}$   and
 $ 0< m^{2}_{2}\leq B  \leq M^{2}_{2},$
    then
 for  every  positive  linear  map  $\Phi$   and  some  positive  real  numbers
  $m_{1}  \leq M_{1}$  and   $m_{2}  \leq M_{2}$
 \begin{eqnarray}\label{2222}
\Phi(A)  \sharp \Phi(B)
\leq   \frac{ \sqrt {M}+\sqrt {m}}{2 \sqrt {Mm}}  \Phi( A \sharp B) ,
\end{eqnarray}
 \end{theorem}
where  $m=\frac{m_{2}}{M_{1}}$   and  $M=\frac{M_{2}}{m_{1}}.$  \\
 Seo  \cite{YSeo}  improved    inequality  above    and  obtained  the  following  inequality:
 \begin{theorem}
 Let  $A,B\in {\mathbb B}({\mathscr H})$ be positive such  that  $ 0< m^{2}_{1}\leq A \leq M^{2}_{1},$  $m^{2}_{2}\leq B \leq M^{2}_{2},$
  $m=\left( \frac {m_{2}}{M_{1}}\right)^{2}$     and
     $M=\left( \frac {M_{2}}{m_{1}}\right)^{2}.$
 Then   for   every
 positive   unital   linear  map  $\Phi $   and     $0\leq \nu \leq 1$
 \begin{eqnarray}\label{2199}
 \Phi(A) \sharp_{\nu} \Phi(B)
 \leq K(m,M,\nu)^{-1}  \Phi(A \sharp_{\nu} B),
 \end{eqnarray}
 where
 $K(m,M,\nu)=\frac{mM^{\nu}-Mm^{\nu}}{(\nu-1)(M-m)}
 \left(  \frac{\nu-1}{\nu}  \frac{M^{\nu}-m^{\nu}}{mM^{\nu}-Mm^{\nu}}  \right)^{\nu}.$
 \end{theorem}

 % The  authors  \cite{moradi}   obtained  another   refinement  of  \eqref{218}:
  % \begin{eqnarray}\label{220}
% \Phi(A)  \sharp \Phi(B)
% \leq   \frac{M+m}{2 \sqrt {Mm K(h)}}  \Phi( A \sharp B).
%  \end{eqnarray}
 In  this   section,   we   give  a   refinement   of     inequality   (\ref{2199}).  To    achieve   this,   we   need   the   following  theorem:

\begin{theorem}\cite{HZuoGShiMFujii}
Suppose that  $A,B\in {\mathbb B}({\mathscr H})$  are positive   and  $m,m^{'},M,M^{'}$ are positive  real  numbers
satisfying    either one of   the  following  conditions:
 \begin{enumerate}
\item
 $ 0<m \leq B \leq m^{'} <M^{'} \leq A \leq M $;
\item
$
0<m \leq A \leq m^{'} <M^{'} \leq B \leq M. $
\end{enumerate}
 Then   for  $ 0 \leq \nu \leq 1$
   \begin{eqnarray}\label{2155}
A \nabla _{v} B \geq  K^{r}(h) (A \sharp_{\nu} B),
\end{eqnarray}
where   $r=\min\{\nu, 1-\nu\},$   $h=\frac{M}{m}$   and   $h^{'}=\frac{M^{'}}{m^{'}}.$
   \end{theorem}
 \begin{theorem}\label{man12}
 Let  $A,B\in {\mathbb B}({\mathscr H})$  such  that  $ 0< m^{2}_{1}\leq A \leq M^{2}_{1},$  $m^{2}_{2}\leq B \leq M^{2}_{2},$
  $m=\left( \frac {m_{2}}{M_{1}}\right)^{2}$     and
     $M=\left( \frac {M_{2}}{m_{1}}\right)^{2}.$   If $M_{1}< m_{2},$
   then   for   every
 positive   unital   linear  map  $\Phi $   and     $0\leq \nu \leq 1$,
 \begin{eqnarray}\label{221}
 \Phi(A) \sharp_{\nu} \Phi(B)
 \leq K(m,M,\nu)^{-1}  K(h)^{-r} \Phi(A \sharp_{\nu} B),
 \end{eqnarray}
 where
 $K(m,M,\nu)=\frac{mM^{\nu}-Mm^{\nu}}{(\nu-1)(M-m)}
 \left(  \frac{\nu-1}{\nu}  \frac{M^{\nu}-m^{\nu}}{mM^{\nu}-Mm^{\nu}}  \right)^{\nu},$
$ r=\min\{\nu,1-\nu\},$  $K(h)=\frac{(1+h)^{2}}{4h}$  and  $h=\frac{m^{2}_{2}}{M^{2}_{1}}.$
Similarly,   one    can  prove  the inequality  for     $M_{2}< m_{1}$    and  $h=\frac{M^{2}_{2}}{m^{2}_{1}}.$
 \end{theorem}
 \begin{proof}
 For   $t\in [m,M].$
 We put
 $F(t)=\nu t^{1-\nu}+(1-\nu) \lambda_{0}  t^{-\nu},$
   where
 $$\mu_{0}=\frac{\nu(M-m)}{M^{\nu}-m^{\nu}}\quad \quad  \lambda_{0}=\frac{\nu}{1-\nu} \frac{M^{1-\nu}-m^{1-\nu}}
 {m^{-\nu}-M^{-\nu}}.$$
 Easy    computation   shows    that
  $\max_{t\in[m,M]} F(t)=F(M)=F(m)$   and  $F(M)=F(m)=\mu_{0}.$
  Hence
  \begin{align}\label{230}
  \nu  t^{1-\nu}+(1-\nu) \lambda_{0} t^{-\nu}  \leq  \mu_{0}.
  \end{align}
Using  the  fact  that
  $ 0< m^{2}_{1}\leq A \leq M^{2}_{1}$  and   $m^{2}_{2}\leq B \leq M^{2}_{2},$    we  get
    $mI \leq C= A^{-\frac{1}{2}}B  A^{-\frac{1}{2}}\leq MI.$    Considering    inequality  (\ref{230})  with  $ C= A^{-\frac{1}{2}}B  A^{-\frac{1}{2}},$
 we    obtain
 \begin{eqnarray*}
\nu C+(1-\nu) \lambda_{0}  I \leq \mu_{0} C^{\nu}.
\end{eqnarray*}
Multiplying   both  sides  of  the  latter  inequality  by   $A^{\frac{1}{2}},$   we   have
\begin{eqnarray}\label{223}
\nu \Phi (B)+(1-\nu) \lambda_{0} \Phi (A) \leq \mu_{0} \Phi (A\sharp_{\nu} B).
\end{eqnarray}
Using  (\ref{2155})    for  two  operators  $ \lambda_{0} \Phi( A)$  and  $\Phi (B)$    yields   that
\begin{eqnarray}\label{224}
\lambda^{1-\nu}_{0} \Phi(A) \sharp_{\nu} \Phi(B)
 \leq  K(h,2)^{-r}  ( \nu \Phi(B)+ (1-\nu)\lambda_{0}\Phi(A)).
 \end{eqnarray}
From  \eqref{223}  and  \eqref{224},  we  obtain  inequality  \eqref{221}.
 \end{proof}

\begin{remark}
Note  that   the   right  side   of    inequality   \eqref{221}  is   a  better   bound   than   inequality   \eqref{2199},  since
the  Kantorovich  constant  $K(h)$  is  increasing  on  the  interval  $(1,+\infty).$
\end{remark}

\begin{remark}
If  we  put  $\nu=\frac{1}{2}$   in  Theorem \ref{man12},   then  we   obtain    a  refinement  of   (\ref{2222}),
 since
the  Kantorovich  constant  $K(h)$  is  increasing  on  the  interval  $(1,+\infty).$
\end{remark}

\textbf{Acknowledgement.}
The  first  author   would  like  to  thank    the  Lorestan  University  and
the second author would like to thank the Tusi Mathematical Research Group (TMRG).

%-----------------------------------------------------------------------------------------

%=================================================================================

%===================================================================================================================================
\bibliographystyle{amsplain}

 \end{document}